\documentclass[12pt]{elsarticle}

\usepackage[margin=1in]{geometry}  % set the margins to 1in on all sides
\usepackage{graphicx}              % to include figures
\usepackage{amsmath}               % great math stuff
\usepackage{amsfonts}              % for blackboard bold, etc
\usepackage{amssymb}
\usepackage{amsthm}                % better theorem environments
\usepackage{mathtools}		% to use coloneqq
\usepackage{enumitem}		% enumeration
\usepackage{xcolor}               

\newif\ifshowchanges
\showchangesfalse
\newcommand{\new}[1]{\ifshowchanges\textcolor{red}{#1}\else#1\fi}
\newcommand{\tru}[1]{\ifshowchanges\textcolor{blue}{#1}\else#1\fi}

\usepackage{hyperref}

\newtheorem{thm}{Theorem}[section]
\newtheorem{lem}[thm]{Lemma}
\newtheorem{prop}[thm]{Proposition}
\newtheorem{cor}[thm]{Corollary}

\theoremstyle{definition}
\newtheorem{defi}[thm]{Definition}
\newtheorem{exa}[thm]{Example}
\newtheorem{rem}[thm]{Remark}

\DeclareMathOperator{\id}{id}
\DeclareMathOperator{\cconvh}{\overline{convh}}
\DeclareMathOperator{\convh}{convh}

\newcommand{\xb}{\mathbf{x}}

\newcommand{\cdb}{\mathcal{C}}

\begin{document}

\nocite{*}

\title{The closed convex hull of all associative copulas is a $Z$-set in the family of all bivariate 
Schur concave copulas, and related results}

\author[WT]{Wolfgang Trutschnig\corref{cor}}
\ead{wolfgang@trutschnig.net}
\address[WT]{Department for Artificial Intelligence and Human Interfaces, Hellbrunnerstrasse 34, University of Salzburg, 5020 Salzburg, Austria}

\cortext[cor]{Corresponding author}

\author[MUF]{Manuel Úbeda-Flores}
\ead{mubeda@ual.es}
\address[MUF]{Department of Mathematics, University of Almería, Carretera de Sacramento s/n, 04120 Almería, Spain}

\begin{abstract}
\tru{Very recently, Problem 2.5 stated by Mesiar and Sarkoci in \cite{MeSa} and originally 
formulated by Durante was answered in the negative by \v{S}eliga \cite{Sel}. The author showed that the  closed convex hull of all bivariate associative copulas is a proper subset of the family of all Schur concave copulas, and even \new{computed the exact uniform distance of his counterexamples to the closed convex hull (so that the Hausdorff distance between the two families is at least $\frac32-\sqrt{2}$).} Using the fact that every bivariate associative copula can be represented as a (finite or countably infinite) ordinal sum of Archimedean copulas and the minimum copula $M$, and working with the Markov kernels of Archimedean copulas, we first give an alternative, simple proof of the strict inclusion. We do so by showing that convex combinations of associative copulas can assign only very little mass to \new{staircase-shaped unions of small squares covering} certain line segments parallel to the main diagonal, a result we later
extend to so-called chains and staircases. Building on the first observation, we then go one step further and show that, from a topological perspective, the closed convex hull of all associative copulas is a very small subset of the Schur concave family, namely a so-called $Z$-set, and that the family of associative copulas is, in turn, a $Z$-set in its closed convex hull.
}
\end{abstract}


\maketitle

\section{Introduction}
It is well-known and straightforward to verify that the widely used family $\cdb^{ar}$ of
bivariate Archimedean copulas is not closed with respect to the uniform metric $d_\infty$ and
that its closure is the family $\cdb^{assoc}$ of associative copulas (see \cite{KMP}).
In accordance with $\cdb^{ar}$, the family $\cdb^{assoc}$ is easily shown not to be convex.
Moreover, according to \cite{Al, Ling}, every associative copula admits a
(finite or countably infinite) ordinal sum representation with Archimedean copulas and
the minimum copula $M$ serving as building blocks.
Since each associative copula $C$ is obviously symmetric and easily verified to be Schur concave 
(see \cite{Alsina84,DuSe03}) 
and since the family $\cdb^{sc}$ of all Schur concave copulas is closed and convex it may seem
natural to conjecture that the closed convex hull of $\cdb^{assoc}$ is
$\cdb^{sc}$, i.e., that $\cconvh(\cdb^{assoc})=\cdb^{sc}$ holds. This very question was
formulated as Problem 2.5 in \cite{MeSa} and goes back to Fabrizio Durante.
Problem 2.5 has recently been answered in the negative by \v{S}eliga \cite{Sel}, who constructed a one-parameter family of Schur concave shuffles of $M$ outside $\cconvh(\cdb^{assoc})$ and determined their distance to $\cconvh(\cdb^{assoc})$. He also proved that every extreme point of $\cconvh(\cdb^{assoc})$ is associative and reformulated the second part of Problem 2.6 in \cite{MeSa} as the question whether the extreme points of $\cconvh(\cdb^{assoc})$ are exactly the associative copulas which are extreme points 
of \tru{the full family of copulas} $\cdb$.
The objective of this note is twofold. Firstly, we give an independent proof of the fact 
that $\cconvh(\cdb^{assoc})$ is a proper subset of
$\cdb^{sc}$, based on a different mechanism: elements of $\cconvh(\cdb^{assoc})$ \tru{assign very little 
mass to \new{staircase-shaped unions of small squares covering} line segments parallel to the main diagonal, whereas Schur concave copulas may assign \new{these sets} much more mass. }
Since the resulting counterexamples are not extreme points of $\cdb$, the extremality argument used in \cite{Sel} does not apply to them. Secondly, and this is our main contribution, we show that 
$\cconvh(\cdb^{assoc})$ is a so-called $Z$-set in $\cdb^{sc}$ and that $\cdb^{assoc}$ is, in turn, a $Z$-set in $\cconvh(\cdb^{assoc})$, \tru{i.e., topologically speaking $\cconvh(\cdb^{assoc})$ is negligible 
in $\cdb^{sc}$ and $\cdb^{assoc}$ is negligible in $\cconvh(\cdb^{assoc})$.} \\
The rest of this note is organized as follows: Section \ref{sec:not} gathers some background
on the involved families of copulas and on Markov kernels, which serve as key tool.
Section \ref{sec:main}
then studies one concrete example of a Schur concave copula which is not an element of
$\cconvh(\cdb^{assoc})$, and, working with tools from infinite-dimensional topology then uses this
example in order to show that $\cconvh(\cdb^{assoc})$ is topologically negligible in $\cdb^{sc}$. 
Section~\ref{sec:chains} generalizes the underlying mass bound to arbitrary staircases, shows that the elements of $\cconvh(\cdb^{assoc})$ do not charge chains off the main diagonal and compares this approach with the one in \cite{Sel}, showing in particular that these necessary conditions are not sufficient. Section~\ref{sec:topo2} contains further topological results, \new{and} Section~\ref{sec:concl} concludes the 
paper.

\section{Notation and Preliminaries}\label{sec:not}
For every integer $N$ we will write $\{1,2,\ldots,N\}=:[N]$. Given an arbitrary
metric space $(S,m)$, the Borel $\sigma$-field on $S$ will be
denoted by $\mathcal{B}(S)$ and $\mathbf{1}_E$ denotes the characteristic/indicator function
of a set $E\subseteq S$. We will write $\lambda$ for the Lebesgue measure on $\mathcal{B}([0,1])$ and
$ \lambda_2$ for the Lebesgue measure on $\mathcal{B}([0,1]^2)$.

In what follows $\cdb$ denotes the family of all bivariate copulas, $d_\infty$ the uniform distance
on $\cdb$. It is well-known that $(\cdb,d_\infty)$ is a compact metric space.
For each copula $C$ the corresponding doubly stochastic measure will be denoted by $\mu_C$,
i.e.\ $\mu_C([0,x]\times[0,y]) = C(x,y)$ for all $x,y\in[0,1]$.
For more background on copulas and doubly stochastic measures see \cite{DuSe} and the references therein.

A copula $C\in\mathcal{C}$ is referred to as associative, if for all $x,y,z\in[0,1]$ we have
that
\begin{equation}\label{eq:assoc}
C\bigl(C(x,y),z\bigr) = C\bigl(x,C(y,z)\bigr)
\end{equation}
holds. As mentioned before, we will let $\cdb^{assoc}$ denote the family of all
associative copulas, it is straightforward to verify that $\cdb^{assoc}$ is closed, so
$(\cdb^{assoc}, d_\infty)$ is a compact metric space.
One important subclass of $\cdb^{assoc}$ is that of Archimedean copulas:
An Archimedean copula $A_\varphi$ is a copula induced by a convex and strictly decreasing
function $\varphi:[0,1]\to[0,\infty]$ with $\varphi(1)=0$, called generator, via
\begin{equation}\label{eq:arch}
A_\varphi(x,y) = \psi\bigl(\varphi(x)+\varphi(y)\bigr), \qquad x,y\in[0,1].
\end{equation}
Thereby $\psi:[0,\infty]\to[0,1]$ denotes the pseudoinverse of $\varphi$, defined by
\begin{equation}\label{eq:psi}
\psi(x) :=
\begin{cases}
\varphi^{-1}(x) & \text{if } x\in[0,\varphi(0+)), \\
0 & \text{if } x \geq \varphi(0+),
\end{cases}
\end{equation}
with $\varphi(x+) := \lim_{t\to x+}\varphi(t)$ denoting the right-hand limit.
Considering that generators are only unique up to a multiplicative constant, in what follows
we always assume that $\varphi$ is normalized in the sense that $\varphi(\frac{1}{2})=1$ holds.
Using this convention generators and Archimedean copulas are in one-to-one correspondence.
\tru{We will refer to $\varphi$ as strict, if and only if $\varphi(0+)=\infty$ and as
non-strict otherwise. $\cdb^{ar,s}$ will denote the family of all strict Archimedean copulas.}

As mentioned before, Archimedean copulas and $M$ are intimately connected with associative copulas:  According to \cite{Al, Ling} every associative copula $C$ admits an
ordinal sum representation in the following sense: there exists some
index set $I\subseteq\mathbb{N}$, non-degenerate intervals $[a_i,b_i], i\in I$, and
copulas $A_i \in\mathcal{C}^{ar}\cup\{M\}$ for every $i \in I$, with the following
properties:
\begin{itemize}
\item[(O1)] $\bigcup_{i \in I} [a_i,b_i]$ is dense in $[0,1]$,
%\footnote{\new{Equality $\bigcup_{i\in I}[a_i,b_i]=[0,1]$ cannot be achieved in general: if, for instance, the set of idempotents of $C$ (i.e., of all $x$ with $C(x,x)=x$) is the Cantor set and $C$ is Archimedean on each removed interval, then every non-degenerate block containing a point of the Cantor set which is not an endpoint of a removed interval would contain idempotents as well as non-idempotents in its interior, which is impossible for blocks as in (O3). Condition (O1) is obtained by choosing as blocks the closures of the connected components of $[0,1]\setminus\{x:C(x,x)=x\}$ and of the interior of $\{x:C(x,x)=x\}$; it is not used in the sequel. Compare also with the formulation of the representation in \cite[Section~2]{Sel}, where the blocks need not cover $[0,1]$.}}
\item[(O2)] the intervals $[a_i,b_i], i\in I$, have pairwise disjoint interior
\item[(O3)] for all $x,y \in [0,1]$ we have that $C(x,y)$ can be written as
 \begin{equation}\label{eq:analytic}
 C(x,y) =
\begin{cases}
a_i + (b_i-a_i)\,A_i \left(\tfrac{x-a_i}{b_i-a_i},\tfrac{y-a_i}{b_i-a_i}\right) & \text{if } (x,y)\in(a_i,b_i)^2 \text{ for some } i \in I \\
M(x,y) & \text{otherwise.}
\end{cases}
 \end{equation}
\end{itemize}
Notice that (contrary to the literature) following \cite{KFT} we explicitly include shrunk copulas 
of $M$ along diagonal blocks.

Recall that a copula $C$ is called Schur concave (see \cite{DP} and the references therein)
if and only if
\begin{equation}\label{eq:sc}
C(x,y) \leq C\big(a x + (1-a) y,\, (1-a) x + a y\big)
\end{equation}
holds for all $x,y,a \in [0,1]$. In other words, a symmetric copula $C$ is Schur concave if 
and only if for every
$b \in [0,1]$ we have that the function
$\delta_b : [\max(-b,\,b-1),\ \min(1-b,\,b)] \rightarrow [0,1]$, defined by
\begin{equation}\label{eq_delta_b}
\delta_b(t):=C(b+t,b-t)
\end{equation}
is non-decreasing on $[\max(-b,\,b-1),0]$ and non-increasing
on $[0,\ \min(1-b,\,b)]$. It follows immediately from \eqref{eq:sc} (with $a=0$) 
that every Schur concave copula is
symmetric, i.e., the transpose $C^t$ of $C$, given by $C^t(x,y):=C(y,x)$, coincides with
$C$.

A Markov kernel from $\mathbb{R}$ to $\mathbb{R}$ is a mapping $K:\mathbb{R}\times\mathcal{B}(\mathbb{R})\to[0,1]$ such that for every fixed $E\in\mathcal{B}(\mathbb{R})$ the mapping $x\mapsto K(x,E)$ is (Borel-)measurable and for every fixed $x\in\mathbb{R}$ the mapping $E\mapsto K(x,E)$ is a probability measure.
Given two real-valued random variables $X,Y$ on a joint
probability space $(\Omega,\mathcal{A},\mathbb{P})$ we say that a Markov kernel $K$ is
a regular conditional distribution of $Y$ given $X$ if
$K(X(\omega),E) = \mathbb{E}(\mathbf{1}_E\circ Y\mid X)(\omega)$ holds $\mathbb{P}$-almost surely for every $E\in\mathcal{B}(\mathbb{R})$.
It is well known that for $X,Y$ as above, a regular conditional distribution of $Y$ given $X$ always exists and is unique for $\mathbb{P}^X$-a.e.\ $x\in\mathbb{R}$,
whereby $\mathbb{P}^X$ denotes the push-forward of $\mathbb{P}$ via $X$, i.e., $\mathbb{P}^X(E) = \mathbb{P}(X^{-1}(E))$ for every $E\in\mathcal{B}(\mathbb{R})$ (see, e.g., \cite{Ka}). \\
In case the random vector $(X,Y)$ has distribution function
$C\in\mathcal{C}$ (restricted to $[0,1]^2$)
in what follows we will let $K_C:[0,1]\times\mathcal{B}([0,1])\to[0,1]$ denote (a version of) the regular conditional distribution of $Y$ given $X$ and refer to it as the \emph{Markov kernel of $C$}. Defining the $x$-section of a set $G\in\mathcal{B}([0,1]^2)$ by $G_x := \{y\in[0,1] : (x,y)\in G\}$, we have the following disintegration formula (see \cite{Ka})
\begin{equation}\label{eq:disint}
\int_{[0,1]} K_C(x,G_x)\,\mathrm{d}\lambda(x) = \mu_C(G),
\end{equation}
hence, in particular,
\begin{equation}\label{eq:disint2}
\int_{[0,1]} K_C(x,E)\,\mathrm{d}\mathbb{P}^X(x) = \int_{[0,1]} K_C(x,E)\,\mathrm{d}\lambda(x) = \lambda(E)
\end{equation}
for every $E\in\mathcal{B}([0,1])$. Moreover, for every $G$ of the form
$G:=[0,x] \times [0,y]$ it follows that
$$
C(x,y)=\int_{[0,x]} K_C(s,[0,y])\mathrm{d} \lambda(s),
$$
so for every fixed $y \in [0,1]$ we have that for $\lambda$-almost every
$x \in [0,1]$ the partial derivative $C_{,1}(x,y)$ of
$C$ with respect to the first coordinate coincides with $K_C(x,[0,y])$.
For more information on conditional expectation and general disintegration we refer to
\cite{Ka}. We will refer to a copula $C$ as completely dependent if and only if
there exists some $\lambda$-preserving transformation $h: [0,1] \rightarrow [0,1]$
(i.e., a transformation fulfilling that the push-forward $\lambda^h$ of $\lambda$ via $h$ coincides with $\lambda$) such that $K(x,E):=\mathbf{1}_E(h(x))$ is a Markov kernel of $C$.
Indicating the dependence on $h$ we will write $C_h$ for the corresponding completely
dependent copula.

Markov kernels were used in \cite{p06} for introducing metrics capable of separating
independence and complete dependence, which, in turn, \new{form} the basis for \new{constructing} dependence measures. Defining $D_p: \cdb \times \cdb \rightarrow [0,1]$ by 
\begin{equation}
D_p(A,B) \;=\; \left( \int_{[0,1]^2} \bigl|K_A(x,[0,y]) - K_B(x,[0,y])\bigr|^{p} \, d\lambda_2(x,y) \right)^{1/p}, \qquad p\in[1,\infty),
\end{equation} 
it can be shown that $D_p$ is a metric for every $p \in [1,\infty)$, that the resulting
metric spaces $(\cdb,D_p)$ are complete and separable, and that (despite not being
equivalent as metrics) all $D_p$ induce the same topology. Following \cite{p06} 
in the sequel we will only work with $D_1$. It can be shown that  
convergence with respect to $D_1$ implies convergence with respect to $d_\infty$ but 
not vice versa (for counterexamples see \cite{p06}). There are, however, special subclasses of 
copulas for which 
convergence with respect to $D_1$ and $d_\infty$ coincide - these families include the \new{Archimedean and the} extreme-value family (see \cite{KFT2}) as well as the associative family $\cdb^{assoc}$ (see \new{\cite{KFT}}).    

Using the fact that according to Rademacher's theorem for each
copula $C$ both partial derivatives $C_{,1}$ and $C_{,2}$ exist $\lambda_2$-almost everywhere,
the property of being Schur concave can easily be translated to Markov kernels.
In fact, a symmetric copula $C$ is Schur concave if and only if
\begin{equation}\label{eq:sc.kernel}
(x-y)\Bigl(K_C(x,[0,y]\new{)}-K_C(y,[0,x])\Bigr)\le 0
\end{equation}
for $\lambda_2$-almost every $(x,y)\in[0,1]^2$.

As last ingredient for proving our main result we need the Markov kernels of Archimedean copulas.
Following \cite{FST} and defining the function $f^0: [0,1] \rightarrow [0,1]$ by
$f^0(x):=\psi(\varphi(0)-\varphi(x))$ (with the convention $\psi(+\infty):=0$)
a version of the Markov kernel $K_{A_\varphi}$ of $A_\varphi$ is given by
\begin{equation}\label{eq:kern.arch}
K_{A_\varphi}(x,[0,y]) =
\begin{cases}
0 & \text{if } x\in(0,1),\ y<f^{0}(x), \\
\tfrac{D^{+}\varphi(x)}{D^{+}\varphi(A_\varphi(x,y))} & \text{if } x\in(0,1),\ y\geq f^{0}(x), \\
1 & \text{if } x\in\{0,1\}.
\end{cases}
\end{equation}
\tru{Notice that for strict Archimedean copulas we have $f^0(x)=0$ for every $x \in (0,1]$, so the 
first line in eq. \eqref{eq:kern.arch} is not needed, and $K_{A_\varphi}(x,\{f^0(x)\})=0$ holds
for every $x \in \new{(0,1)}$. According to \new{\cite{KFT2,KFT}}, $\cdb^{ar,s}$ is dense in $(\cdb^{ar},D_1)$ 
and in  $(\cdb^{assoc},D_1)$. \new{Since convergence with respect to $D_1$ implies uniform convergence, $\cdb^{ar,s}$ is also dense in $(\cdb^{ar},d_\infty)$ (see also \cite{KMP}).}}

\section{Main results}\label{sec:main}
We consider the following family of Schur concave copulas.
\begin{exa}\label{ex:Ca}
For every $a \in [\frac{1}{2},1)$ define the $\lambda$-preserving
function $T_a: [0,1] \rightarrow [0,1]$ by
$$
T_a(x) = (x+a)\,\mathbf{1}_{[0,1-a)}(x) + x\,\mathbf{1}_{[1-a,a]}(x) + (x-a)\,\mathbf{1}_{(a,1]}(x)
$$
and let $C_{T_a}$ denote the completely dependent copula corresponding to $T_a$.
It is straightforward to verify that the copula $C_a$, given by
\begin{equation}\label{eq:def.Ca}
C_a=\tfrac{1}{2} (C_{T_a} + M)
\end{equation}
is Schur concave. In fact, $C_a$ is symmetric (since $T_a\circ T_a=\id$ holds $\lambda$-almost 
everywhere) and each function $\delta_b$ defined according to eq. \eqref{eq_delta_b}
is piecewise linear, non-decreasing on $[\max(-b,\,b-1),0]$ and non-increasing
on $[0,\ \min(1-b,\,b)]$. Moreover, obviously the support of $\mu_{C_a}$ consists of three
line segments, one being the diagonal $\{(x,x): \, x \in [0,1]\}$.
\end{exa}
In what follows we work with the copula $C^*:=C_{\frac{1}{2}}$ and proceed via two lemmas which
use the following sets. For every integer $N \geq 3$ and every
$i \in \{0,1,\ldots,2^{N-1}\}$ define $u_i:=\tfrac{1}{2} + \tfrac{i}{2^N}, \,v_i:=\tfrac{i}{2^N}$ and set
\begin{align}
Q^{2^N}&:=\bigcup_{i=1}^{2^{N-1}} [u_{i-1},u_i] \times [v_{i-1},v_i], \,\,\,\,
O^{2^N}:=\bigcup_{i=1}^{2^{N-1}} (u_{i-1},u_i) \times (v_{i-1},v_i).
\end{align}
Notice that for every $N \geq 3$ the set $Q^{2^N}$ covers the line segment connecting
$(\frac{1}{2},0)$ and $(1,\frac{1}{2})$, which is part of the support of $\mu_{C^*}$
(see Figure \ref{fig:sc_example} for an illustration).
Moreover, letting $\partial S$ denote the boundary of a set $S \subseteq [0,1]^2$, considering
that $Q^{2^N}$ is
the union of finitely many squares, each copula $C \in \cdb$ fulfills
$\mu_C(\partial Q^{2^N})=0$, i.e., $Q^{2^N}$ is a continuity set of each doubly stochastic measure
and
$$
\mu_C \left(Q^{2^N}\right)= \mu_C \left(O^{2^N}\right)
$$
holds. Using this property we can prove the following lemma:
\begin{lem}\label{lem:1}
Let $N \geq 3$ be arbitrary but fixed and suppose that $C^* \in \cconvh(\cdb^{assoc})$.
Then there exists some associative copula $A \in \cdb^{assoc}$ fulfilling
\begin{equation}\label{ineq.implication}
\mu_{A}\left(Q^{2^N}\right) \geq \tfrac{1}{4}(1-\tfrac{1}{2^N}).
\end{equation}
\end{lem}
\begin{proof}
By assumption there exist copulas $C_1,C_2,C_3,\ldots \in \convh(\cdb^{assoc})$ converging
uniformly to $C^*$. Translating to the corresponding doubly
stochastic measures we have that the sequence $(\mu_{C_n})_{n \in \mathbb{N}}$ converges
weakly to $\mu_{C^*}$. Considering $\mu_{C^*}(\partial Q^{2^N})=0$ and using
Portmanteau theorem directly yields
\begin{align*}
\lim_{n \rightarrow \infty} \mu_{C_n}\left(Q^{2^N}\right) = \mu_{C^*}\left(Q^{2^N}\right)=
\tfrac{1}{4},
\end{align*}
so there exists some index $n_0 \in \mathbb{N}$ such that
$\mu_{C_n}\left(Q^{2^N}\right) \geq \frac{1}{4}(1-\frac{1}{2^N})$ for every $n \geq n_0$. \\
Fix some $n \geq n_0$. Then by construction there exists some integer $m_n \in \mathbb{N}$,
associative copulas $A_1,A_2,\ldots, A_{m_n} \in \cdb^{assoc}$ and weights
$\alpha_1,\ldots,\alpha_{m_n} \in [0,1]$ fulfilling $\sum_{i=1}^{m_n} \alpha_i=1$, so that
$C_n=\sum_{i=1}^{m_n} \alpha_i A_i$ holds.
Considering $\mu_{C_n}\left(Q^{2^N}\right) \geq \frac{1}{4}(1-\frac{1}{2^N})$
we can therefore find (at least) one $i \in [m_n]$ with
$$
\mu_{A_i}\left(Q^{2^N}\right) \geq \tfrac{1}{4}(1-\tfrac{1}{2^N})
$$
and the proof is complete.
\end{proof}

\begin{figure}[htbp]
  \centering
  \includegraphics[width=\textwidth]{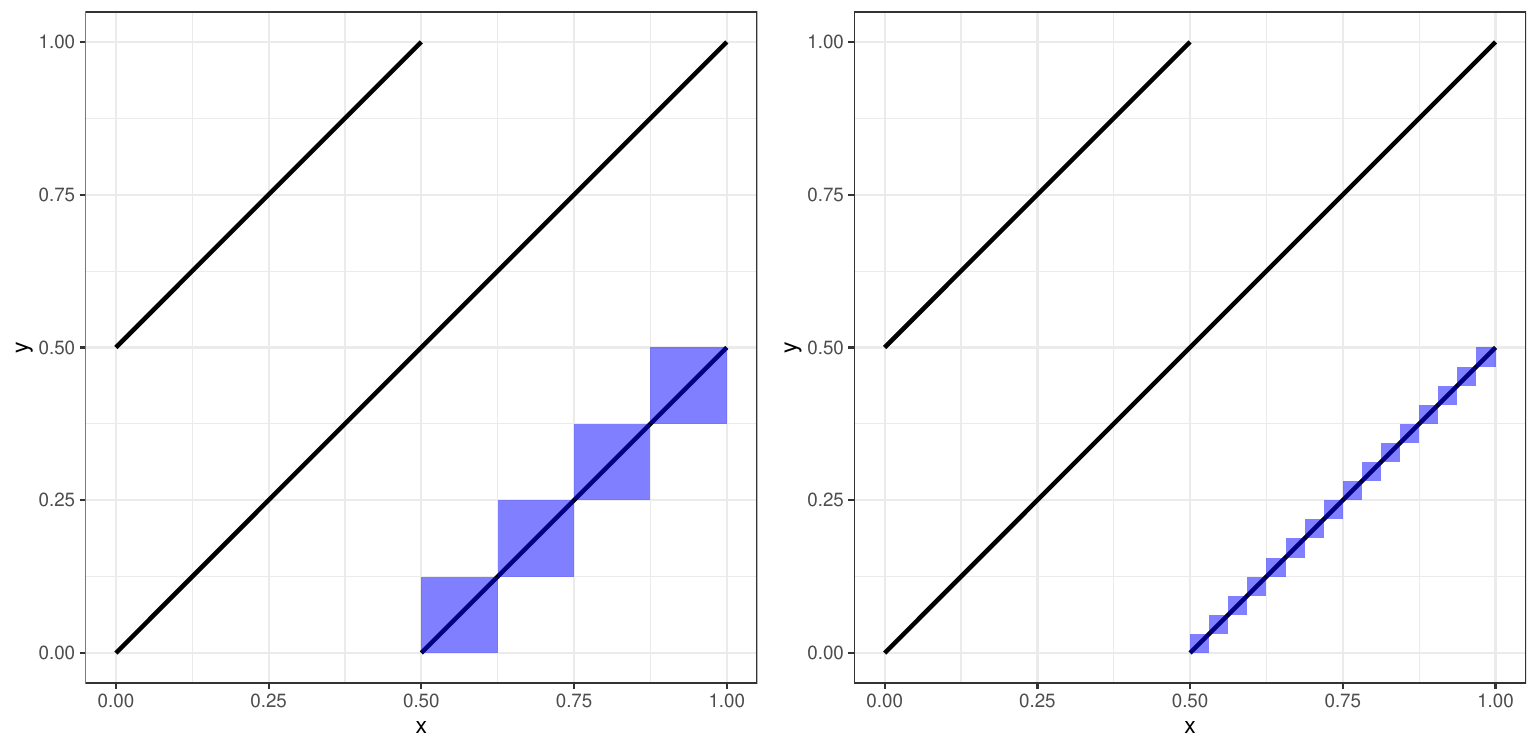}
  \caption{Black: support of the copula $C^*$; blue: the set $Q^{2^N}$ for $N=3$ (left panel)
  and $N=5$ (right panel).}
  \label{fig:sc_example}
\end{figure}

\begin{lem}\label{lem:2}
Every $A \in \cdb^{assoc}$ fulfills
\begin{equation}\label{ineq:little.squares}
\mu_A\left(Q^{2^N}\right) \leq \tfrac{1}{2^N}
\end{equation}
for every $N \geq 3$.
\end{lem}
\begin{proof}
We first reduce the problem to the \tru{strict} Archimedean family $\cdb^{ar,s}$ (step 1) and 
then prove the
corresponding mass concentration result there (step 2). \\
(Step 1) As mentioned in Section \ref{sec:not}, every associative copula $A$ can be
expressed according to eq. \eqref{eq:analytic}.
Suppose that $(a_i,b_i), i \in I \subseteq \mathbb{N}$
with $a_i<b_i$ for every $i \in I$
denotes the intervals corresponding to the ordinal sum, i.e., $\mu_A$ restricted
to the square $(a_i,b_i)^2$ coincides with the push-forward of some Archimedean copula
$A_{\varphi_i}$
or $M$ under the affine similarity $T_i(x,y)=(a_i+x(b_i-a_i), a_i+y(b_i-a_i))$, weighted by
$b_i-a_i$. Obviously at most one block $(a_i,b_i)^2$ can have non-empty intersection with $Q^{2^N}$
- this block has to fulfill $a_i < \frac{1}{2} < b_i$. 
\tru{Outside these squares $\mu_A$ is concentrated on the diagonal, which does not intersect $Q^{2^N}$.}
Since the case where no block intersects $Q^{2^N}$ and the case where the only block
intersecting $Q^{2^N}$ carries a shrunk copy $(b_i-a_i)\mu^{T_i}_M$ of $M$ are trivial
(in both cases the mass is $0$ so ineq. \eqref{ineq:little.squares} certainly holds),
it suffices to focus on $(b_i-a_i)\mu_{A_\varphi}^{T_i}$ with $A_\varphi \in \cdb^{ar}$.
Simplifying notation we denote the unique block by $(a,b)^2$, set $T:=T_i$, and let $A_\varphi$ denote the corresponding
Archimedean copula with normalized generator $\varphi$.
 We want to show that $\mu_{A_\varphi}^{T}$ fulfills
\begin{equation}\label{ineq:arch}
\mu_A\left(Q^{2^N}\right)=(b-a)\,\mu_{A_\varphi}\left(T^{-1}\left(Q^{2^N}\right)\right)=
(b-a)\,\mu_{A_\varphi}^T\left(Q^{2^N}\right) \leq \tfrac{1}{2^N}.
\end{equation}
\tru{\new{Doing} so, using denseness of $\cdb^{ar,s}$ in $(\cdb^{ar},d_\infty)$ (mentioned at the end of Section 2) \new{and the fact that $A_\varphi\mapsto\mu_{A_\varphi}(T^{-1}(Q^{2^N}))$ is continuous with respect to $d_\infty$ (since $T^{-1}(Q^{2^N})\cap[0,1]^2$ is a finite union of rectangles),} it is enough to prove ineq. \eqref{ineq:arch} for strict $A_\varphi$. } \\
(Step 2) Define $f: \mathbb{R} \rightarrow \mathbb{R} $ by $f(x)=\frac{x-a}{b-a}$ and set
$$
\hat{u}_i:=\min\{f(u_i),1\}, \hat{v}_i:=\max\{f(v_i),0\}
$$
for every $i \in \{0,1,\ldots,2^{N-1}\}$.
Then considering
\begin{equation}
\hat{Q}^{2^N}:= \bigcup_{i=1}^{2^{N-1}}[\hat{u}_{i-1},\hat{u}_i] \times
[\hat{v}_{i-1},\hat{v}_i], \,\,\,\,
\hat{O}^{2^N}:=\bigcup_{i=1}^{2^{N-1}} (\hat{u}_{i-1},\hat{u}_i) \times (\hat{v}_{i-1},\hat{v}_i).
\end{equation}
we obviously have $\mu_C(T^{-1} (Q^{2^N} ) \cap (0,1)^2)=\mu_C(\hat{Q}^{2^N})=
\mu_C(\hat{O}^{2^N})$ for every copula $C$, so in particular for $A_\varphi$.
Using the convention $\frac{b}{\infty}=0$
for every $b > 0$, applying disintegration, and considering the Markov kernel
$K_{A_\varphi}$ of $A_\varphi$ as derived in \cite{FST} it follows that
\begin{align*}
\mu_{A_\varphi}\left(T^{-1}\left(Q^{2^N}\right)\right)&=
\mu_{A_\varphi}\left(\hat{O}^{2^N}\right)
= \sum_{i=1}^{2^{N-1}} \int_{(\hat{u}_{i-1},\hat{u}_i)} K_{A_\varphi}\left(x,(\hat{v}_{i-1},\hat{v}_i) \right) \,d\lambda(x) \\
&\leq \sum_{i=1}^{2^{N-1}} \int_{(\hat{u}_{i-1},\hat{u}_i)} K_{A_\varphi}\left(x,(\hat{v}_{i-1},\hat{v}_i] \right) \,d\lambda(x) \\
&= \sum_{i=1}^{2^{N-1}} \int_{(\hat{u}_{i-1},\hat{u}_i)} K_{A_\varphi}\left(x,[0,\hat{v}_{i}]\right) -
K_{A_\varphi}\left(x,[0,\hat{v}_{i-1}] \right) \,d\lambda(x) \\
&= \sum_{i=1}^{2^{N-1}} \int_{(\hat{u}_{i-1},\hat{u}_i)} \vert D^+\varphi(x)\vert
\left( \tfrac{1}{\vert D^+ \varphi\vert (A_\varphi(x,\hat{v}_{i})) } -
\tfrac{1}{\vert D^+ \varphi\vert (A_\varphi(x,\hat{v}_{i-1})) }\right) \,d\lambda(x). 
\end{align*}
Considering the facts that $\vert D^+\varphi\vert $ is non-increasing, that $A_\varphi$
is monotonically increasing in both arguments, that $\hat{u}_i > 0$, and that
$$
\vert D^+\varphi(x) \vert \leq \vert D^+\varphi(\hat{u}_0) \vert < \infty
$$
holds for every $x \in [\hat{u}_0,1]$, altogether we obtain
\begin{align*}
\mu_{A_\varphi}\left(T^{-1}\left(Q^{2^N}\right)\right)&\leq
\sum_{i=1}^{2^{N-1}} \int_{(\hat{u}_{i-1},\hat{u}_i)} \vert D^+\varphi(x)\vert
\left( \tfrac{1}{\vert D^+ \varphi\vert (A_\varphi(\hat{u}_{i},\hat{v}_{i})) } -
\tfrac{1}{\vert D^+ \varphi\vert (A_\varphi(\hat{u}_{i-1},\hat{v}_{i-1})) }\right) \,d\lambda(x) \\
&\leq \sum_{i=1}^{2^{N-1}}\left( \tfrac{1}{\vert D^+ \varphi\vert (A_\varphi(\hat{u}_{i},\hat{v}_{i})) } -
\tfrac{1}{\vert D^+ \varphi\vert (A_\varphi(\hat{u}_{i-1},\hat{v}_{i-1})) }\right)
\vert D^+\varphi(\hat{u}_0)\vert \, \tfrac{1}{2^N} \tfrac{1}{b-a} \\
&= \vert D^+\varphi(\hat{u}_0)\vert \, \tfrac{1}{2^N}\tfrac{1}{b-a}\, \sum_{i=1}^{2^{N-1}}\left( \tfrac{1}{\vert D^+ \varphi\vert (A_\varphi(\hat{u}_{i},\hat{v}_{i})) } -
\tfrac{1}{\vert D^+ \varphi\vert (A_\varphi(\hat{u}_{i-1},\hat{v}_{i-1})) }\right) \\
&=  \vert D^+\varphi(\hat{u}_0)\vert \, \tfrac{1}{2^N} \tfrac{1}{b-a}
\left( \tfrac{1}{\vert D^+ \varphi\vert (A_\varphi(1,\hat{v}_{2^{N-1}})) } -
\tfrac{1}{\vert D^+ \varphi\vert (A_\varphi(\hat{u}_{0},0)) }\right) \\
&\leq \vert D^+\varphi(\hat{u}_0)\vert \, \tfrac{1}{2^N} \tfrac{1}{b-a}
 \tfrac{1}{\vert D^+ \varphi(\hat{v}_{2^{N-1}})\vert}  \\
&\leq  \tfrac{1}{2^N} \tfrac{1}{b-a}
\end{align*}
Having this directly yields eq. \eqref{ineq:arch} and the proof is complete.
\end{proof}

\begin{thm}\label{thm:proper}
The family $\cconvh(\cdb^{assoc})$ is a proper subset of $\cdb^{sc}$.
In particular, the copula $C^*$ is not an element of $\cconvh(\cdb^{assoc})$.
\end{thm}
\begin{proof}
Direct consequence of the previous two lemmas and the fact that for $N \geq 3$ we have
$\tfrac{1}{4}(1-\tfrac{1}{2^N}) > \frac{1}{2^N}$.
\end{proof}
The first assertion of Theorem~\ref{thm:proper} has recently also been obtained by \v{S}eliga \cite{Sel} by means of a Schur concave shuffle of $M$; the mechanism behind our proof is a different one, see Section~\ref{sec:chains}.

\begin{rem}
The copula $C^*=C_{\frac{1}{2}}$ was chosen in order to avoid unnecessary computations
- in this case each corner of the set $Q^{2^N}$ is of the form $\frac{i}{2^N}$.
Nevertheless, proceeding ana\-lo\-gously yields that each of the copulas $C_a$ defined
according to eq. \eqref{eq:def.Ca} is Schur concave but not an element of $\cconvh(\cdb^{assoc})$.
In other words: the aforementioned construction yields uncountably many elements of
$\cdb^{sc} \setminus \cconvh(\cdb^{assoc})$.
\end{rem}

We now go one step further and prove that from a topological perspective
$\cconvh(\cdb^{assoc})$ is a very small subset of $\cdb^{sc}$. Doing so, we will use some
tools from infinite-dimensional topology, which we first recall for the sake of completeness
(for more background we refer to the textbooks \cite{Sakai2020,vanMill1989}).
$\mathcal{H}=[-1,1]^\mathbb{N}$ will denote the Hilbert cube, elements of
$\mathcal{H}$ will be denoted by bold symbols. We endow $\mathcal{H}$ with
the metric $\rho$, given
by $\rho(\mathbf{x},\mathbf{y})=\sum_{i=1}^\infty \frac{1}{2^i}\vert x_i-y_i \vert $.
It is straightforward to verify that $\rho$ is a metrization of coordinate-wise convergence
and that $(\mathcal{H},\rho)$ is a convex and compact metric space.
The Hilbert cube $(\mathcal{H},\rho)$ is usually regarded as `prototype' compact metric space,
since every compact metric space $(\Omega,m)$ is homeomorphic to a compact subset
of $(\mathcal{H},\rho)$ (\cite{vanMill1989}).

\begin{defi}
A subset $A$ of a metric space $(\Omega,m)$ is called a \emph{$Z$-set} if it is closed and for every
continuous mapping $\varepsilon: \Omega \rightarrow (0,1)$ there exists some continuous function
$f: \Omega \rightarrow \Omega\setminus A$ such that $m(x,f(x)) < \varepsilon(x)$ holds for all
$x \in \Omega$.
\end{defi}
In the case $(\Omega,m)$ is compact, it obviously suffices to consider $\varepsilon$ constant.
Naturally we interpret $Z$-sets as topologically negligible (and hence topologically small) sets.
It is straightforward to verify that every $Z$-set in $(\Omega,m)$ is nowhere dense (and hence of
first Baire category) in $(\Omega,m)$, and that closed subsets of $Z$-sets are $Z$-sets.
Moreover, every singleton in $\mathcal{H}$ is a $Z$-set and
for every fixed $i \in \mathbb{{N}}$
so is the set $\{\xb \in \mathcal{H}: x_i \in \{-1,1\}\}$.

The following theorem will be key for showing that $\cconvh(\cdb^{assoc})$ is a very small subset
of the compact metric space $(\cdb^{sc},d_\infty)$ - the proof of a more general case involving
finitely many sets $S_i$ can be found in \cite{DFST}.

\begin{thm}\label{thm:main.tool}
Suppose that $S_1 \subset S_2 $ are convex, compact subsets
of a normed linear space $(E, \Vert \cdot \Vert)$, that $S_1$ is homeomorphic to the Hilbert cube
$\mathcal{H}$, and that $S_{1}$ is a $Z$-set in $S_2$.
Then there exists a homeomorphism $h : [0,1] \times \mathcal{H} \to S_2$ fulfilling
\begin{align*}
h(\{0\}\times\mathcal{H}) &= S_1.
\end{align*}
\end{thm}

\noindent We want to apply this result to the situation of $E$ being
the Banach space $(C([0,1]^2),\Vert \cdot \Vert_\infty)$
of all continuous real-valued functions on $[0,1]^2$, $S_1=\cconvh(\cdb^{assoc})$ and $S_2=\cdb^{sc}$.
Obviously both sets $\cconvh(\cdb^{assoc})$ and $\cdb^{sc}$ are convex and compact subsets of
$(C([0,1]^2),\Vert \cdot \Vert_\infty)$, so in order to assure applicability of
Theorem \ref{thm:main.tool} it suffices to prove the following lemma:

\begin{lem}\label{lem}
$\cconvh(\cdb^{assoc})$ is a $Z$-set in $\cdb^{sc}$ and
$(\cconvh(\cdb^{assoc}),d_\infty)$ is homeomorphic to $(\mathcal{H},\rho)$.
\end{lem}
\begin{proof}
(i) We start with the second assertion. According to the Keller--Dobrowolski--Toru\'n\-czyk theorem
(see \cite[Theorem 2.11, Remark 2.12]{DFST} and original references therein) it suffices
to show that $\cconvh(\cdb^{assoc})$ contains infinitely many linearly independent elements
(linearly independent in the Banach space $(C([0,1]^2),\Vert \cdot \Vert_\infty)$), which
can be done as follows:
For every $a \in (0,1)$ set $I:=\{1,2\}, a_1=0,b_1=a_2=a,b_2=1$ as well as $A_1=W,A_2=M$, and
let $O_a$ denote the corresponding ordinal sum defined according to eq. \eqref{eq:analytic}.
Then the family $\{O_a: \, a \in (0,1)\}$ is linearly independent.
In fact, suppose that for $n \in \mathbb{N}$ and $0< a_1 < a_2 < \ldots < a_n < 1$
there are constants $e_1,\ldots,e_n \in \mathbb{R}$ such that
$\sum_{i=1}^n e_i O_{a_i}= \mathbf{0} \in C([0,1]^2)$. Then the finite signed measure
$\vartheta$, defined by
$$
\vartheta(G)=\sum_{i=1}^n e_i \mu_{O_{a_i}}(G),\quad G \in \mathcal{B}([0,1]^2)
$$
is the zero-measure on $\mathcal{B}([0,1]^2)$.
For the set $\Gamma_{a_1}:=\{(t,a_1-t): t \in [0,a_1]\} \in \mathcal{B}([0,1]^2)$ we therefore get
$$
0=\vartheta(\Gamma_{a_1}) = \sum_{i=1}^n e_i \mu_{O_{a_i}}(\Gamma_{a_1}) = e_1 \mu_{O_{a_1}}(\Gamma_{a_1}) = e_1a_1,
$$
implying $e_1=0$.
Consi\-dering the sets $\Gamma_{a_i}:=\{(t,a_i-t): t \in [0,a_i]\} \in \mathcal{B}([0,1]^2)$ and
proceeding inductively eventually shows $e_1=e_2=\ldots=e_n=0$. \\
(ii) For showing the first assertion first notice that proceeding as in the proof of the two lemmas
and using Portmanteau theorem we have that every copula $C \in \cconvh(\cdb^{assoc})$ fulfills
\begin{equation}\label{ineq:little.squares2}
\mu_C\left(Q^{2^N}\right) \leq \tfrac{1}{2^N}
\end{equation}
for every $N \geq 3$. Having this, we can now show that $\cconvh(\cdb^{assoc})$ is
a $Z$-set in $\cdb^{sc}$ as follows: Let $\varepsilon \in (0,1)$ be arbitrary but fixed, choose
$N \geq 7$ sufficiently large so that $\frac{1}{2^{N-4}}<\varepsilon$ holds, and define the
continuous mapping
$f: \cdb^{sc} \rightarrow \cdb^{sc}$ by
$$
f(C) = \tfrac{1}{2^{N-3}} C^* \, + \, \left(1- \tfrac{1}{2^{N-3}}\right) C.
$$
Then for every $C \in \cdb^{sc}$ we have
\begin{align*}
\mu_{f(C)}\left(Q^{2^N}\right) &=  \tfrac{1}{2^{N-3}} \mu_{C^*}\left(Q^{2^N}\right) \, + \,
\left(1- \tfrac{1}{2^{N-3}}\right) \mu_C\left(Q^{2^N}\right) \\
&\geq \tfrac{1}{2^{N-3}} \mu_{C^*}\left(Q^{2^N}\right) = \tfrac{1}{2^{N-3}} \tfrac{1}{4} =
\tfrac{1}{2^{N-1}} \\
&> \tfrac{1}{2^N},
\end{align*}
implying $f(C) \in \cdb^{sc} \setminus \cconvh(\cdb^{assoc})$. In other words: $f$ maps
$\cdb^{sc}$ into $\cdb^{sc} \setminus \cconvh(\cdb^{assoc})$, and it remains to
show that $f$ is $\varepsilon$-close to the identity. Using the triangle inequality, however,
it follows immediately that
\begin{align*}
d_\infty(f(C),C) &\leq \tfrac{2}{2^{N-3}} =  \tfrac{1}{2^{N-4}} < \varepsilon
\end{align*}
for every $C \in \cdb^{sc}$, and the proof is complete.
\end{proof}
Having this, applying Theorem \ref{thm:main.tool} now directly yields the following
slightly surprising result without further ado (also see \cite{Liu2024a}):
\begin{thm}\label{thm:topo}
There exists a homeomorphism $h : [0,1] \times \mathcal{H} \to \cdb^{sc}$ fulfilling
\begin{align*}
h(\{0\}\times\mathcal{H}) &= \cconvh(\cdb^{assoc}).
\end{align*}
\end{thm}
\noindent \tru{In other words: if we identify \new{$\cdb^{sc}$ with the Hilbert cube $[0,1]\times\mathcal{H}$ via $h$}, then (geometrically speaking) $\cconvh(\cdb^{assoc})$ corresponds to \new{the face $\{0\}\times\mathcal{H}$}.}

\section{Staircases and chains}\label{sec:chains}

In this section we show that the mechanism behind Lemma~\ref{lem:2} yields a general necessary condition for membership in $\cconvh(\cdb^{assoc})$ (Theorem~\ref{thm:chains}), which provides a large class of Schur concave copulas outside $\cconvh(\cdb^{assoc})$ (Corollary~\ref{cor:outside}), and we compare this approach with the one in \cite{Sel} (Remark~\ref{rem:seliga} and Proposition~\ref{prop:flip}).

The proof of Lemma~\ref{lem:2} only uses that $Q^{2^N}$ is a `staircase' of rectangles below the main diagonal. In what follows, a \emph{staircase} is a set of the form
\begin{equation}\label{eq:staircase}
R(\mathbf{s},\mathbf{t}):=\bigcup_{i=1}^m[s_{i-1},s_i]\times[t_{i-1},t_i]
\end{equation}
with $m\in\mathbb{N}$ and $0\le t_0\le t_1\le\cdots\le t_m\le s_0<s_1<\cdots<s_m\le1$; we write $\Vert\mathbf{s}\Vert:=\max_{i\in[m]}(s_i-s_{i-1})$. 
\tru{Notice that analogous to the sets $Q^{2^N}$ considered
in the previous section, each staircase $R(\mathbf{s},\mathbf{t})$ is a $\mu_C$ continuity set
for each copula $C \in \cdb$, i.e., $\mu_C(\partial R(\mathbf{s},\mathbf{t}))=0$ holds.} 
Proceeding as in the proof of Lemma~\ref{lem:2} (replacing $u_i,v_i$ by $s_i,t_i$, the point $\frac12$ by $s_0$ and the length $\frac{1}{2^N}$ by $\Vert\mathbf{s}\Vert$, and using $t_m\le s_0$ instead of $v_{2^{N-1}}=u_0$) yields the following generalization.

\begin{prop}\label{prop:staircase}
Every $A\in\cdb^{assoc}$ fulfills $\mu_A\big(R(\mathbf{s},\mathbf{t})\big)\le\Vert\mathbf{s}\Vert$ for every staircase $R(\mathbf{s},\mathbf{t})$. The bound is sharp: if $m=1$ and $s_1-s_0=t_1-t_0$, equality holds for the ordinal sum with the single summand $W$ on $[t_0,s_1]$.
\end{prop}

Since $C\mapsto\mu_C\big(R(\mathbf{s},\mathbf{t})\big)$ is a finite linear combination of values of 
$C$ \tru{(in fact, a linear combination of $C$ evaluated at the 
vertices of $R(\mathbf{s},\mathbf{t})$), it is affine and continuous 
with respect to $d_\infty$,} so the inequality extends to $\cconvh(\cdb^{assoc})$. 
\tru{This has the following consequence \new{for} so-called `chains'.} Recall that $L\subseteq[0,1]^2$ is 
\tru{referred to as a} chain (with respect to the coordinatewise order) if $(x-x')(y-y')\ge0$ for all $(x,y),(x',y')\in L$; typical examples are line segments with positive slope and, more generally, graphs of non-decreasing functions. We write $\Delta:=\{(x,x):x\in[0,1]\}$ 
\tru{for the main diagonal of $[0,1]^2$. Figure \ref{fig:staircase_example} depicts
two examples of staircases (blue) and chains (magenta).} 

\begin{figure}[htbp]
  \centering
  \includegraphics[width=\textwidth]{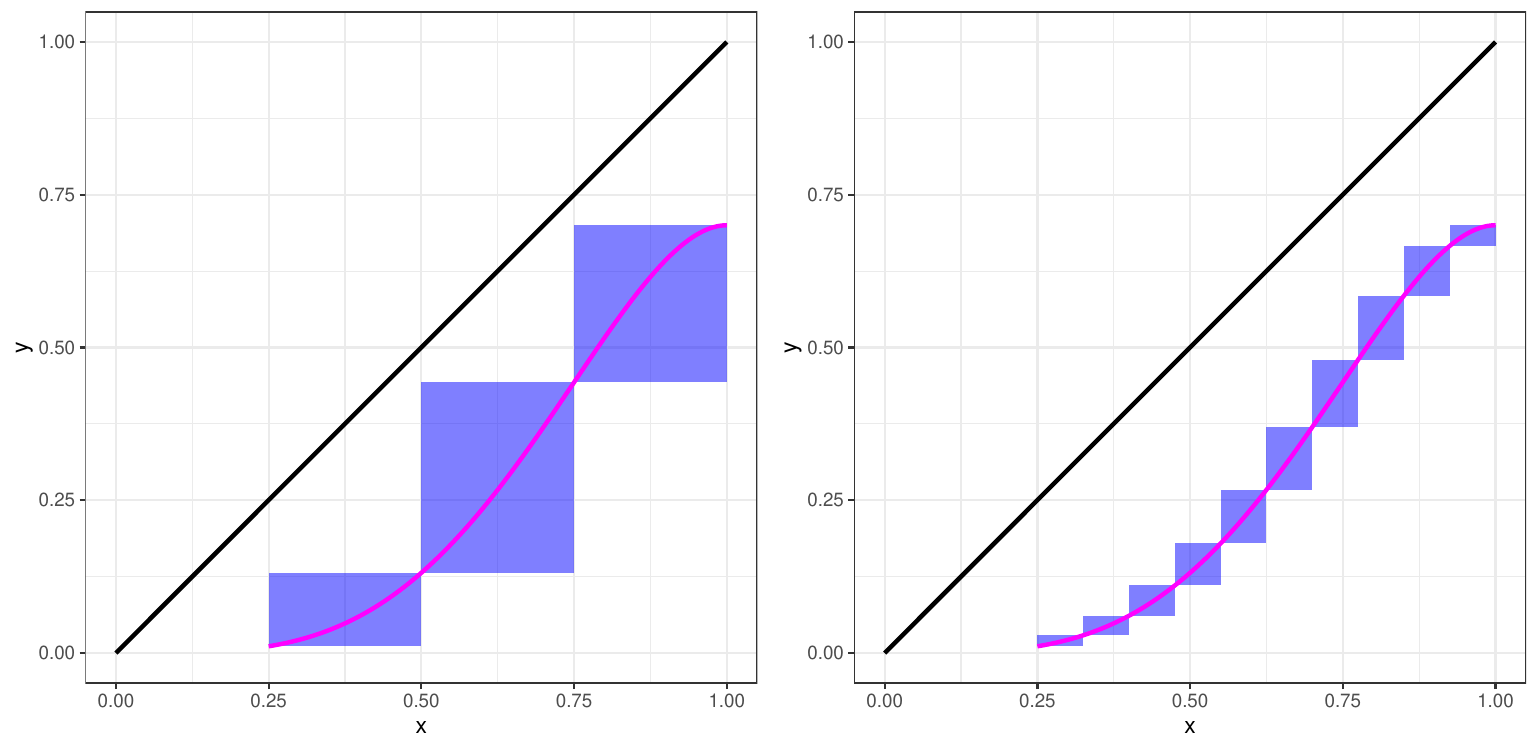}
  \caption{Two examples of staircases (blue) and chains (magenta). 
  Letting \new{$R_1$ and $R_2$} denote the left and the right staircase, respectively, according to Proposition \ref{prop:staircase}, every associative copula $A$ fulfills \new{$\mu_A(R_1)$} $\leq \frac{1}{4}$ as well as \new{$\mu_A(R_2)$} $\leq \frac{1}{40}$.}
  \label{fig:staircase_example}
\end{figure}

\begin{thm}\label{thm:chains}
\tru{For every $C\in\cconvh(\cdb^{assoc})$ and every staircase $R(\mathbf{s},\mathbf{t})$ we have  
$\mu_C\big(R(\mathbf{s},\mathbf{t})\big)\le\Vert\mathbf{s}\Vert$. Moreover, $\mu_C(L)=0$ for every 
chain $L \in \mathcal{B}([0,1]^2)$ fulfilling $L\subseteq[0,1]^2\setminus\Delta$.}
\end{thm}
\begin{proof}
The first assertion has already been shown. Since $C$ is symmetric, it suffices to 
\tru{establish the result for chains $L\subseteq\{(x,y)\in [0,1]^2:\, x>y\}$. 
Considering that every such $L$ is the 
 union of the sets $L_q:=L\cap([q,1]\times[0,q])$, $q\in\mathbb{Q}\cap(0,1)$,} it suffices to show $\mu_C(L_q)=0$, \tru{which can be done as follows}: 
 Fix $m\in\mathbb{N}$, set $s_i:=q+\frac{i}{m}(1-q)$, \tru{define $B_i$ by 
$$
B_i:=\{y \in [0,1]: \exists x \in (s_{i-1},s_i) \text{ with } (x,y) \in L_q \}, 
$$ 
and} set $t_0:=0$ and $t_i:=\max\{t_{i-1},\sup B_i\}$ (with $\sup\emptyset:=0$). The chain property implies $\sup B_j\le\inf B_i$ for $j<i$, hence $B_i\subseteq[t_{i-1},t_i]$, so $L_q$ is contained in the staircase $R(\mathbf{s},\mathbf{t})$ up to finitely many vertical segments, which are $\mu_C$-null sets.
\tru{This shows $\mu_C(L_q)\le\frac{1-q}{m}$ for every $m\in\mathbb{N}$, which, considering 
$m \rightarrow \infty$}, completes the proof.
\end{proof}

The following simple observation shows that half of the mass of a Schur concave copula can be distributed in an essentially arbitrary way.

\begin{lem}\label{lem:halfM}
Let $B$ be \tru{an arbitrary symmetric} copula.  \tru{Then for every $\alpha\in[0,\frac12]$} 
the copula $C:=\alpha B+(1-\alpha)M$ is Schur concave. The constant $\frac12$ is best possible.
\end{lem}
\begin{proof}
Obviously $C$ is symmetric. For $b\in[0,1]$ and $0\le t<t'\le\min(b,1-b)$ we have 
$$
B(b+t,b-t)-B(b+t',b-t')\ge-\mu_B\big((b+t,b+t']\times[0,1]\big)=-(t'-t)
$$
\tru{as well as} $M(b+t,b-t)-M(b+t',b-t')=t'-t$. \tru{This implies  
$\delta_b(t)-\delta_b(t')\ge(1-2\alpha)(t'-t)\ge0$, so $\delta_b$ is non-increasing. 
Since $b$ was arbitrary, this shows that \new{$C$} is Schur concave.} 
Notice that symmetry of $B$ has not been used here, so even for non-symmetric $B$ all functions 
$\delta_b$ have the required monotonicity although $C$ is not symmetric. 
For $B=C_{T_{1/2}}$ and $\alpha>\frac12$ we get $\delta_{1/2}(t)=\alpha t+(1-\alpha)(\frac12-t)$ for 
$t\in[0,\frac14]$, \tru{so on this interval $\delta_{1/2}$ is strictly increasing, and 
$C$ is not Schur concave.}
\end{proof}

\begin{cor}\label{cor:outside}
Let $B$ be a symmetric copula with \new{$\mu_B(L)>0$} for some \tru{Borel-measurable} chain $L\subseteq[0,1]^2\setminus\Delta$. Then $\alpha B+(1-\alpha)M\in\cdb^{sc}\setminus\cconvh(\cdb^{assoc})$ for every $\alpha\in(0,\frac12]$.
\end{cor}
\begin{proof}
Direct consequence of Lemma~\ref{lem:halfM} and Theorem~\ref{thm:chains}, \tru{since in this case we have} 
\new{$\mu_{\alpha B+(1-\alpha)M}(L)\ge\alpha\,\mu_B(L)>0$}.
\end{proof}

The copulas $C_a$ from Example~\ref{ex:Ca} are of this form, since $\mu_{C_{T_a}}$ assigns mass $1-a$ to the segment $\{(x,x-a):x\in(a,1]\}$, which is a chain not intersecting $\Delta$. Further examples are obtained by choosing for $B$ any symmetric shuffle of $M$ whose support contains a segment of 
slope $1$ off the diagonal.

\tru{Our} approach also yields a quantitative statement \tru{without additional effort}: 
writing $\Phi(C):=\mu_C(Q^{8})$ as a linear combination of values of $C$ (nine values with coefficients $\pm1,\pm2$ whose absolute values sum up to $12$, plus the constant $\frac18$) shows $\vert\Phi(C)-\Phi(D)\vert\le12\,d_\infty(C,D)$. Since $\Phi(C^*)=\frac14$ and $\Phi\le\frac18$ on $\cconvh(\cdb^{assoc})$, we obtain $d_\infty\big(C^*,\cconvh(\cdb^{assoc})\big)\ge\frac{1}{96}$. For the shuffles of $M$ 
\tru{as considered} in \cite{Sel}, the exact distance to $\cconvh(\cdb^{assoc})$ has been determined in \cite[Theorem~5.5]{Sel}.

\begin{rem}\label{rem:seliga}
The chain \tru{property} in Theorem~\ref{thm:chains} \tru{does not exclude} a Schur concave completely dependent copula $C_h$ (such as the shuffles of $M$ in \cite{Sel}) from $\cconvh(\cdb^{assoc})$, since such $C_h$ assigns zero mass to every chain $L\subseteq[0,1]^2\setminus\Delta$. 
\tru{In fact, suppose that $C_h$ is a \new{Schur concave (hence symmetric)} completely dependent copula with $\mu_{C_h}(L)>0$} for some Borel measurable chain $L\subseteq\{(x,y):x>y\}$. 
\tru{Symmetry of $C_h$ implies that $h\circ h=\id$ holds $\lambda$-a.e., the set 
$E:=\{x \in [0,1]:(x,h(x))\in L\}$ has positive \tru{Lebesgue} measure 
(since $\mu_{C_h}$ assigns all mass to the graph of $h$),
} and $h$ is non-decreasing on $E$ with $h(x)<x$. 
Considering \tru{that $\lambda$-almost every $x \in E$ is a Lebesgue point of 
$\mathbf{1}_E$ (see \cite{Ru}), using the fact that $h|_E$ has at most countably many discontinuities, and applying disintegration to the set $S:=\{(x,x')\in E^2:x<x',\,h(x')<x\}$ it follows that
\begin{align*}
\lambda_2(S)&= \int_E \lambda(S_x) d\lambda(x) = \int_{E \cap \text{Cont(h,E)}} \lambda(S_x) d\lambda(x)
>0, 
\end{align*}
where $\text{Cont(h,E)}$ denotes the continuity points of $h$ restricted to $E$. 
For $\lambda_2$-a.e. $(x,x')\in S$ the point $(u,v):=(x,h(x'))$ fulfills $u>v$, $K_{C_h}(u,[0,v])=\mathbf{1}_{[0,h(x')]}(h(x))=1$ as well as \new{$K_{C_h}(v,[0,u])=\mathbf{1}_{[0,u]}(h(v))=\mathbf{1}_{[0,x]}(x')=0$}. 
Since the mapping $(x,x')\mapsto(x,h(x'))$ is $\lambda_2$-preserving and $\lambda_2$-a.e.\ injective, 
this contradicts \eqref{eq:sc.kernel}.} Conversely, \tru{none of the} copulas in Corollary~\ref{cor:outside} are  completely dependent. \\
Hence the chain \tru{property} in Theorem~\ref{thm:chains} cannot recover the counterexamples in \cite{Sel}, and the two families of counterexamples are disjoint. Schur concave shuffles of $M$ are, however, covered by the extremality argument in \cite{Sel}: an extreme point of $\cdb$ which belongs to $\cconvh(\cdb^{assoc})$ is an extreme point of $\cconvh(\cdb^{assoc})$ and therefore associative by \cite[Theorem~3.4]{Sel}. Conversely, the copula $C^*$ from Section~\ref{sec:main} is not an extreme point of $\cdb$ and fulfills the star-shapedness inequality $C^*(t,1-t)\ge2t\,C^*(\frac12,\frac12)$, $t\in[0,\frac12]$, of \cite[Lemma~5.3]{Sel}, so neither the extremality argument nor the separating functionals used in \cite{Sel} exclude $C^*$ from $\cconvh(\cdb^{assoc})$.
\end{rem}
\tru{In the subsequent proposition $C_\theta^*$ with $\theta \in (0,\frac{1}{2})$ 
denotes the completely dependent copula 
$C_{h_\theta}$, whose corresponding $\lambda$-preserving transformation $h_\theta$ is given by
$$
h_\theta(x) =
\begin{cases}
  1-x, & x\in[0,\theta]\cup[1-\theta,1], \\
  x,  & \text{otherwise}.
\end{cases}
$$
According to \new{\cite[Theorems~4.7 and 5.1]{Sel}} $C_\theta^*$ is Schur concave but fulfills $C_\theta^* \notin\cconvh(\cdb^{assoc})$ \new{(the copulas $C^*_\theta$ should not be confused with the copula $C^*=C_{\frac12}$ from Section~\ref{sec:main})}. 
}
\begin{prop}\label{prop:flip}
\tru{For every $\theta\in(0,\frac12)$ the copula $C_\theta^*$ fulfills both conditions in Theorem~\ref{thm:chains}. 
Consequently, the necessary conditions in Theorem~\ref{thm:chains} are not sufficient for membership in 
$\cconvh(\cdb^{assoc})$.}
\end{prop}
\begin{proof}
Off the main diagonal, \tru{$\mu_{C_\theta^*}$ assigns all mass to} the two decreasing segments 
$$
S_1:=\{(x,1-x):x\in[0,\theta]\},\quad S_2:= \{(x,1-x):x\in[1-\theta,1]\}.
$$
\tru{Every chain intersects $S_1$ and $S_2$ in at most one point. Moreover, any staircase 
$R(\mathbf{s},\mathbf{t})$ meets the diagonal in at most one point and, apart from finitely many points, 
 intersects $S_1$ and \new{$S_2$} within a single rectangle $[s_{i-1},s_i]\times[t_{i-1},t_i]$, whose intersection with the segments carries mass of at most $s_i-s_{i-1}$.} Hence both conditions in Theorem~\ref{thm:chains} hold.
\end{proof}

\section{Further topological results}\label{sec:topo2}
\tru{This section complements Theorem~\ref{thm:topo} in two directions:} Firstly, we derive some consequences for the family $\cdb^{sc}\setminus\cconvh(\cdb^{assoc})$ of all Schur concave copulas which are not contained in the closed convex hull of $\cdb^{assoc}$ (Corollary~\ref{cor:topo}). 
Our results considerably strengthen the observation in \cite[Remark~5.2]{Sel} stating 
that this family contains the interior of a topo\-logi\-cal arc. Secondly, we show that $\cdb^{assoc}$ is, in turn, topologically negligible in $\cconvh(\cdb^{assoc})$ (Theorem~\ref{thm:assocZ}). 
\tru{Together} with the fact that all extreme points of $\cconvh(\cdb^{assoc})$ are associative \cite[Theorem~3.4]{Sel}, this yields information on the set of extreme points of $\cconvh(\cdb^{assoc})$ (Corollary~\ref{cor:assocZ}).

\tru{In what follows we will refer to the elements of a subset $A$ of a metric space $(S,m)$
as `topologically typical' (in the sense of Baire categories), if and only if $A$ is 
co-meager, i.e., if $A$ is the complement of a set of first Baire category.} 

\begin{cor}\label{cor:topo}
\tru{The following assertions hold:}
\begin{enumerate}[label=(\roman*)]
\item $\cdb^{sc}\setminus\cconvh(\cdb^{assoc})$ is open and dense in $(\cdb^{sc},d_\infty)$; in particular, in the sense of Baire categories a typical Schur concave copula is not an element of $\cconvh(\cdb^{assoc})$.
\item $\cdb^{sc}\setminus\cconvh(\cdb^{assoc})$ is homeomorphic to $(0,1]\times\mathcal{H}$.
\item For every metric space \tru{$(S,m)$}, every continuous mapping 
$\Lambda: S \to\cdb^{sc}$ and 
every $\varepsilon>0$ there exists a continuous mapping $\tilde\Lambda: S \to\cdb^{sc}\setminus\cconvh(\cdb^{assoc})$ with $d_\infty(\Lambda(x),\tilde\Lambda(x))<\varepsilon$ for all $x\in S$.
\tru{If} $\Lambda$ is a topological embedding, then so is $\tilde\Lambda$.
\end{enumerate}
\end{cor}
\begin{proof}
\tru{Assertion} (i) follows from Lemma~\ref{lem}, since $Z$-sets are closed and nowhere dense\new{.} \tru{Assertion} (ii) holds since the homeomorphism $h$ from Theorem~\ref{thm:topo} maps $(0,1]\times\mathcal{H}$ onto $\cdb^{sc}\setminus\cconvh(\cdb^{assoc})$. 
\tru{To show the third \new{statement}, consider} $\tilde\Lambda:=f\circ\Lambda$ with $f$ as in the proof of Lemma~\ref{lem}; $f$ is affine and injective on the compact set $\cdb^{sc}$, and 
hence a homeomorphism onto its range.
\end{proof}
Statement (iii) means that every parametric family of Schur concave copulas can be perturbed by an arbitrarily small amount (uniformly in the parameter and preserving injectivity) to a family which completely avoids $\cconvh(\cdb^{assoc})$. \\

We now show that $\cdb^{assoc}$ \tru{is topologically negligible} in its closed convex hull. 
\tru{The following lemma will be key.} 

\begin{lem}\label{lem:diag}
Every $A\in\cdb^{assoc}$ fulfills $\mu_A(\Delta)=\lambda\big(\{x\in[0,1]:A(x,x)=x\}\big)$.
\end{lem}
\begin{proof}
Let $A_\varphi$ be \tru{an arbitrary Archimedean copula} and $x\in(0,1)$. 
\tru{According to} eq.~\eqref{eq:kern.arch}, the probability measure $K_{A_\varphi}(x,\cdot)$ can only have atoms at $y=f^0(x)$ and at points $y$ for which $A_\varphi(x,y)$ is one of the (countably many) discontinuity points of $D^+\varphi$. Since $f^0$ is non-increasing and $x\mapsto A_\varphi(x,x)$ is strictly increasing where positive, $K_{A_\varphi}(x,\{x\})>0$ holds for at most countably many $x$, so disintegration yields $\mu_{A_\varphi}(\Delta)=0$. \\
For general $A\in\cdb^{assoc}$ the assertion now follows from the ordinal sum representation, since Archimedean \tru{copulas fulfill $A_\varphi(x,x) < x$ for all $x \in (0,1)$} and since outside the Archimedean blocks $\mu_A$ coincides with $\mu_M$.
\end{proof}

\begin{thm}\label{thm:assocZ}
Let $\mathcal{F}\subseteq\cdb$ be closed and convex with $\cdb^{assoc}\subseteq\mathcal{F}$. Then $\cdb^{assoc}$ is a $Z$-set in $(\mathcal{F},d_\infty)$. In particular, $\cdb^{assoc}$ is a $Z$-set in $\cconvh(\cdb^{assoc})$, in $\cdb^{sc}$ and in $\cdb$.
\end{thm}
\begin{proof}
\tru{Consider} $\varepsilon\in(0,1)$ and set $P:=\frac12(\Pi+M)\in\mathcal{F}$. The continuous mapping
$g$, \tru{given by $g(C):=(1-\varepsilon)C+\varepsilon P$,} maps $\mathcal{F}$ into $\mathcal{F}$ 
and obviously fulfills $d_\infty(g(C),C)\le\frac{\varepsilon}{2}$. For every $C\in\mathcal{F}$ we have $\mu_{g(C)}(\Delta)\ge\frac{\varepsilon}{2}\mu_M(\Delta)>0$, whereas $g(C)(x,x)\le(1-\varepsilon)x+\frac{\varepsilon}{2}(x^2+x)<x$ for every $x\in(0,1)$. 
\tru{Applying Lemma~\ref{lem:diag} therefore directly yields $g(C)\notin\cdb^{assoc}$ and the proof is 
complete.}
\end{proof}

\begin{cor}\label{cor:assocZ}
\tru{The following assertions hold:}
\begin{enumerate}[label=(\roman*)]
\item \tru{A topologically typical} element of $\cconvh(\cdb^{assoc})$ is not associative.
\item The closure of the set of extreme points of $\cconvh(\cdb^{assoc})$ is a $Z$-set in $\cconvh(\cdb^{assoc})$; in particular, it is nowhere dense.
\item There exists a homeomorphism $\Psi:\cconvh(\cdb^{assoc})\to\cdb^{sc}$ with $\Psi(A)=A$ for every $A\in\cdb^{assoc}$.
\end{enumerate}
\end{cor}
\begin{proof}
\tru{Assertion} (i) is an immediate consequence of Theorem~\ref{thm:assocZ}.  According to \cite[Theorem~3.4]{Sel}, every extreme point of $\cconvh(\cdb^{assoc})$ is associative, so the closure of the set of extreme points is a closed subset of the $Z$-set $\cdb^{assoc}$. \tru{This shows Assertion (ii)}. \\
\tru{Finally}, both $\cconvh(\cdb^{assoc})$ and $\cdb^{sc}$ are homeomorphic to $\mathcal{H}$ (Lemma~\ref{lem} and \cite{Liu2024a}), and $\cdb^{assoc}$ is a $Z$-set in both of them by Theorem~\ref{thm:assocZ}. \tru{The desired} assertion therefore follows from Anderson's $Z$-set unknotting theorem (see \cite{vanMill1989}), applied to the identity on $\cdb^{assoc}$.
\end{proof}

We note that Lemma~\ref{lem:diag} also yields a necessary condition related to \cite[Open Question~1]{Sel}, which asks whether a non-trivial convex combination of two distinct associative copulas can be associative: if $A,B\in\cdb^{assoc}$, $t\in(0,1)$ and $tA+(1-t)B$ is associative, then the sets $I_A$ and $I_B$ of idempotents of $A$ and $B$ coincide up to a $\lambda$-null set. In fact, the set of idempotents of $tA+(1-t)B$ is $I_A\cap I_B$, so applying Lemma~\ref{lem:diag} to $A$, $B$ and $tA+(1-t)B$ yields $\lambda(I_A\cap I_B)=t\lambda(I_A)+(1-t)\lambda(I_B)$, which is only possible if $\lambda(I_A)=\lambda(I_B)=\lambda(I_A\cap I_B)$.

\section{Conclusions}\label{sec:concl}
We have given an alternative negative answer to Problem 2.5 in \cite{MeSa}, which is based on
Markov kernels of Archimedean copulas: the mass that an associative copula can assign to a staircase of rectangles below the main diagonal is bounded by the mesh of the staircase (Lemma~\ref{lem:2} and
 Proposition~\ref{prop:staircase}). Consequently, no element of $\cconvh(\cdb^{assoc})$ 
\tru{assigns \new{positive} mass to a} chain outside the main diagonal (Theorem~\ref{thm:chains}), whereas half of the mass of a Schur concave copula can be placed on such a chain (Lemma~\ref{lem:halfM}). 
This \tru{very observation} yields a large and explicit class of Schur concave copulas outside $\cconvh(\cdb^{assoc})$ (Corollary~\ref{cor:outside}). This obstruction is complementary to the one used in \cite{Sel}: it never excludes Schur concave shuffles of $M$ such as the copulas in \cite{Sel}, whereas the copulas in Corollary~\ref{cor:outside} are not extreme points of $\cdb$, so the extremality argument of \cite{Sel} does not apply to them (Remark~\ref{rem:seliga}). In particular, the conditions in Theorem~\ref{thm:chains} are necessary but not sufficient for membership in $\cconvh(\cdb^{assoc})$ (Proposition~\ref{prop:flip}).

From a topological \tru{\new{perspective}}, $\cconvh(\cdb^{assoc})$ is a very small subset of $\cdb^{sc}$: it is a $Z$-set, so the pair $\big(\cdb^{sc},\cconvh(\cdb^{assoc})\big)$ is homeomorphic to $\big([0,1]\times\mathcal{H},\{0\}\times\mathcal{H}\big)$ (Theorem~\ref{thm:topo}), and the Schur concave copulas outside $\cconvh(\cdb^{assoc})$ form an open dense subset of $\cdb^{sc}$ which is homeomorphic to $(0,1]\times\mathcal{H}$ (Corollary~\ref{cor:topo}). In the same spirit, $\cdb^{assoc}$ is a $Z$-set in every closed convex family of copulas containing it (Theorem~\ref{thm:assocZ}), so the closure of the set of extreme points of $\cconvh(\cdb^{assoc})$ is nowhere dense in $\cconvh(\cdb^{assoc})$ (Corollary~\ref{cor:assocZ}). A complete description of these extreme points (Problem 2.6 in \cite{MeSa}; see \cite{Sel} for a reformulation of its second part) remains open.

\section*{Acknowledgements}
\noindent The first author gratefully acknowledges the support of the WISS project
`IDA-lab Salzburg’ (20204-WISS/225/197-2019 and 20102-F1901166 KZP).


\begin{thebibliography}{99}

\bibitem{Alsina84}
\new{C. Alsina, On Schur-concave t-norms and triangle functions, in: General Inequalities 4 (Oberwolfach, 1983), Internat. Schriftenreihe Numer. Math. 71, Birkh\"auser, Basel, 1984, pp. 241-248.}

\bibitem{Al}
C. Alsina, B. Schweizer, M. Frank, Associative Functions, World Scientific Publishing, Singapore, 2006.

\bibitem{DFST}
F. Durante, J. Fern\'andez S\'anchez, W. Trutschnig, On families of copulas and quasi-copulas homeomorphic to the Hilbert cube and related results, preprint submitted to arXiv.

\bibitem{DP}
F. Durante, P.L. Papini, A weakening of Schur-concavity for copulas, Fuzzy Sets Syst. 158 (2007) 1378-1383.

\bibitem{DuSe03}
\new{F. Durante, C. Sempi, Copul\ae\ and Schur-concavity, Internat. Math. J. 3 (2003) 893-905.}

\bibitem{DuSe}
F. Durante, C. Sempi, Principles of Copula Theory, CRC Press, Boca Raton FL, 2016.

\bibitem{FST}
J. Fern\'andez S\'anchez, W. Trutschnig, Singularity aspects of Archimedean copulas, J. Math. Anal. Appl. 432 (2015) 103-113.

\bibitem{Ka}
O. Kallenberg, Foundations of Modern Probability, Springer, New York, 2002.

\bibitem{KFT}
T. Kasper, S. Fuchs, W. Trutschnig, On convergence of associative copulas and related results, Depend. Model. 9 (2021) 307-326.

\bibitem{KFT2}
T. Kasper, S. Fuchs, W. Trutschnig, On weak conditional convergence of bivariate Archimedean and Extreme Value copulas, and consequences to nonparametric estimation, Bernoulli 27 (4) (2021) 2217-2240.

\bibitem{KMP}
E.P. Klement, R. Mesiar, E. Pap, Uniform approximation of associative copulas by strict and non-strict copulas, Illinois J. Math. 45 (4) (2001) 1393-1400.

\bibitem{Ling}
C.H. Ling, Representation of associative functions, Publ. Math. Debrecen 12 (1965) 189-212.

\bibitem{Liu2024a}
D. Liu, The topological space of Schur-concave copulas is homeomorphic to the Hilbert cube, Inform. Sci. 678 (2024) 120909.

\bibitem{MeSa}
R. Mesiar, P. Sarkoci, Open problems posed at the tenth International Conference on Fuzzy Set Theory and Applications (FSTA 2010, Liptovsk\'y J\'an, Slovakia), Kybernetika 46 (4) (2010) 585-599.

\bibitem{Ru}
\tru{W. Rudin, Real and Complex Analysis, \new{3rd edition,} McGraw-Hill International \new{Editions}, Singapore, 1987.}

\bibitem{Sakai2020}
K. Sakai, Topology of Infinite-Dimensional Manifolds, Springer Monographs in Mathematics, Springer, Singapore, 2020.

\bibitem{Sel}
\new{A. \v{S}eliga, A Schur-concave shuffle of $M$ outside the closed convex hull of associative copulas, Iranian J. Fuzzy Syst. 23 (6) (2026) 223-238.}

\bibitem{p06}
\tru{W. Trutschnig, On a strong metric on the space of copulas and its induced dependence measure, J. Math. Anal. Appl. 384 (2011) 690-705.}

\bibitem{vanMill1989}
J. van Mill, Infinite-Dimensional Topology. Prerequisites and Introduction, North-Holland Publishing Co., Amsterdam, 1989.

\end{thebibliography}
\end{document}